\documentclass[12pt]{amsart}
\usepackage{amssymb,hyperref}
\usepackage{graphicx} 

\newtheorem{theorem}{Theorem}[section]
\newtheorem{definition}[theorem]{Definition}
\newtheorem{proposition}[theorem]{Proposition}
\newtheorem{lemma}[theorem]{Lemma}

\begin{document}

\title[Block-transposed Wishart matrices]{Asymptotic eigenvalue distributions\\of block-transposed Wishart matrices}

\author{Teodor Banica}
\address{T.B.: Department of Mathematics, Cergy-Pontoise University, 95000 Cergy-Pontoise, France. {\tt teodor.banica@u-cergy.fr}}

\author{Ion Nechita}
\address{I.N.: CNRS, Laboratoire de Physique Th\'eorique, IRSAMC, Universit\'e de Toulouse, UPS, 31062 Toulouse, France. {\tt nechita@irsamc.ups-tlse.fr}}


\begin{abstract}
We study the partial transposition ${W}^\Gamma=(\mathrm{id}\otimes \mathrm{t})W\in M_{dn}(\mathbb C)$ of a Wishart matrix $W\in M_{dn}(\mathbb C)$ of parameters $(dn,dm)$. Our main result is that, with $d\to\infty$, the law of $m{W}^\Gamma$ is a free difference of free Poisson laws of parameters $m(n\pm 1)/2$. Motivated by questions in quantum information theory, we also derive necessary and sufficient conditions for these measures to be supported on the positive half line. 
\end{abstract}

\maketitle

\section*{Introduction}

The partial transposition of a $d \times d$ block matrix $W\in M_d(\mathbb C)\otimes M_n(\mathbb C)$ is the matrix ${W}^\Gamma$ obtained by transposing each of the $n \times n$ blocks of $W$. That is, we let ${W}^{\Gamma}=(\mathrm{id}\otimes \mathrm{t})W$, where $\mathrm{id}$ is the identity of $M_d(\mathbb C)$, and $\mathrm{t}$ is the transposition of $M_n(\mathbb C)$. The partial transposition operation can be defined using coordinates, as follows. A particular decomposition $\mathbb C^{dn} =\mathbb C^d\otimes\mathbb C^n$ induces a decomposition $M_{dn}(\mathbb C)=M_d(\mathbb C)\otimes M_n(\mathbb C)$. The entries of a matrix $W \in M_d(\mathbb C)\otimes M_n(\mathbb C)$ can be indexed by four indices, $i,j\in\{1,\ldots,d\}$ identifying the block of $A$ the matrix entry belongs to, and $a,b\in\{1,\ldots,n\}$ fixing the position of the matrix entry inside the block. Then, the partial transposed matrix $W^\Gamma$ has coordinates
$${W}^{\Gamma}_{ia,jb}=W_{ib,ja}.$$

Motivated by questions in quantum information theory, the partial transposition operation for Wishart matrices was studied by Aubrun \cite{aub}, who showed that, for certain special values of the parameters, the empirical spectral distribution of ${W}^{\Gamma}$ converges in moments to a non-centered semicircular distribution. In this paper we discuss the general case, our main result being as follows.

\medskip

\noindent {\bf Theorem A.} {\em Let $W$ be a complex Wishart matrix of parameters $(dn,dm)$. Then, with $d\to\infty$, the empirical spectral distribution of $m{W}^{\Gamma}$ converges in moments to a free difference of free Poisson distributions of respective parameters $m(n\pm 1)/2$.}

\medskip

Observe that if one applies the transposition map on the first factor of the tensor product $M_d(\mathbb C)\otimes M_n(\mathbb C)$ instead of the second one, the spectral distribution of ${W}^{\Gamma}$ remains unchanged. Indeed, this follows from $(\mathrm{t}\otimes \mathrm{id}) = (\mathrm{t}\otimes \mathrm{t}) \circ (\mathrm{id}\otimes \mathrm{t})$ and from the fact that applying the global transposition leaves the spectral distribution unchanged.

The above theorem basically generalizes Aubrun's result in \cite{aub}, modulo some standard free probability facts, regarding the relationship between the free Poisson laws and the semicircle laws. Also, as a remark of algebraic interest, in the formal limit $m=t/n\to 0$, the law that we obtain is a free difference of free Poisson laws having the same parameter, namely $t/2$, so it is a (modified) free Bessel law in the sense of \cite{bb+}.

Our second result regards the various properties of the limiting measure. By using tools from free probability theory and from classical analysis, we obtain the following result.

\medskip

\noindent {\bf Theorem B.} {\em The limiting measure found in Theorem A has the following properties:
\begin{enumerate}
\item It has at most one atom, at $0$, of mass $\max\{1-mn, 0\}$.

\item It has positive support iff $n\leq m/4+1/m$ and $m \geq 2$.
\end{enumerate}}

\medskip

The main motivation for the above results comes from quantum information theory. In quantum information theory, the partial transposition map is known to be an \emph{entanglement witness}: it allows to test if a quantum state (represented by a positive, unit trace matrix) is entangled, in the following sense. If a bi-partite quantum state $\rho \in M_d(\mathbb C)\otimes M_n(\mathbb C)$ is separable (i.e. it can be written as a convex combination of product states $\rho^{(1)}_i \otimes \rho^{(2)}_i$), then its partial transposition $\rho^\Gamma$ is also a quantum state. However, if $\rho$ is entangled, then $\rho^\Gamma$ may fail to be positive. In the case where $\rho^\Gamma$ is a positive matrix, the quantum state $\rho$ is said to be PPT (Positive Partial Transpose). Hence, separable states are always PPT and non-PPT states are necessarily entangled. The equivalence of entanglement and non-PPT is known to hold only for total dimension smaller than 6 ($2\times 2$ or $2 \times 3$ product systems) and it fails for larger dimensions, in the sense that there exist PPT entangled states. In the same spirit as in \cite{aub}, the results regarding the positivity of the support of the limit measure can be interpreted as results about typicality of PPT states for large quantum systems. Wishart matrices (normalized to have unit trace) are known to be physically reasonable models for random quantum states on a tensor product $\mathbb C^d \otimes \mathbb C^n$, the parameter $m$ of the Wishart distribution being related to the size of some environment $\mathbb C^{dm}$ needed to define the state. So, as a conclusion, Theorem A and Theorem B above indicate that, when $m>2$ and $n<m/4+1/m$, a typical state in $\mathbb C^d \otimes \mathbb C^n$ is PPT.

The paper is organized as follows: in \ref{sec:moments} we present a combinatorial formula for the asymptotic moments of $m{W}^{\Gamma}$, in \ref{sec:mgf} we find an equation for the corresponding moment generating function, and in \ref{sec:main} we state and prove the main result. In sections \ref{sec:examples} and \ref{sec:prop} we discuss some of the properties of the asymptotic eigenvalue distributions (examples, atoms, density, positivity), and the final section \ref{sec:final} contains a few concluding remarks.

\subsection*{Acknowledgments}

We would like to thank the ANR projects Galoisint and Granma for their financial support during the accomplishment of the present work. Part of this work was done when I.N. was a postdoctoral fellow at the University of Ottawa where he was supported by NSERC discovery grants and an ERA.

\section{Formula for moments}\label{sec:moments}

We recall that a complex Wishart matrix $W$ of parameters $(dn, dm)$ is defined as $W = (dm)^{-1} GG^*$, where $G$ is a $dn \times dm$ matrix with i.i.d. complex Gaussian $\mathcal N(0,1)$ entries. Consider the block-transpose matrix ${W}^{\Gamma}$, constructed in the introduction. We denote by $\mathrm{tr}:M_{dn}(\mathbb C)\to\mathbb C$ the normalized trace, $\mathrm{tr}(W)=(dn)^{-1}\sum W_{ii}$.

We recall a number of results from the combinatorial theory of noncrossing partitions; see \cite{nsp} for a detailed presentation of the theory. For a permutation $\sigma \in S_p$, we introduce the following standard notation:
\begin{itemize}
\item $\# \sigma$ is the number of cycles of $\sigma$;
\item $|\sigma|$ is its length, defined as the minimal number $k$ such that $\sigma$ can be written as a product of $k$ transpositions. The function $(\sigma, \pi) \to |\sigma^{-1}\pi|$ defines a distance on $S_p$. One has $\#\sigma + |\sigma| = p$.
\item $\mathrm{e}(\sigma)$ is the number of blocks of \emph{even} size of $\sigma$.
\end{itemize}
Let $\gamma\in S_p$ be a fixed arbitrary full cycle (an element of order $p$ in $S_p$). The set of permutations $\sigma \in S_p$ which saturate the triangular inequality $|\sigma| + |\sigma^{-1} \gamma| = |\gamma| = p-1$ is in bijection with the set $NC(p)$ of noncrossing partitions of $\{1,\ldots,p\}$. We call such permutations \emph{geodesic} and we shall not distinguish between a non crossing partition and its associated geodesic permutation. We also recall a well known bijection between $NC(p)$ and the set $NC_2(2p)$ of noncrossing \emph{pairings} of $2p$ elements. To a noncrossing partition $\pi \in NC(p)$ we associate an element $\mathrm{fat}(\pi) \in NC_2(2p)$ as follows: for each block $\{i_1, i_2, \ldots, i_k\}$ of $\pi$, we add the pairings $\{2i_1-1, 2i_k\}, \{2i_1, 2i_2-1\}, \{2i_2, 2i_3-1\}, \ldots, \{2i_{k-1}, 2i_k-1\}$ to $\mathrm{fat}(\pi)$. The inverse operation is given by collapsing the elements $2i-1, 2i \in \{1, \ldots, 2p\}$ to a single element $i \in \{1, \ldots, p\}$. The number of blocks of $\pi$ and $\mathrm{fat}(\pi)$ are related by $\# \pi = \#(\mathrm{fat}(\pi) \vee \rho_{12})$, where $\vee$ is the join operation on the lattice $NC_2(2p)$ and $\rho_{12} = (12)(34)\cdots(2p-1\, 2p)$ is the fattened identity permutation. Similarly, $ \#(\pi \gamma) = \#(\mathrm{fat}(\pi) \vee \rho_{14})$, where $\rho_{14}$ is the pairing that corresponds to the fattening of the inverse full cycle $\gamma^{-1}(i) = i-1$. More precisely, $\rho_{14}$ pairs an element $2i$ with $2(i-1)-1 = 2i-3$, or, equivalently, an element $i \in \{1, \ldots, 2p\}$ with $i+(-1)^{i+1}3$.

The following combinatorial lemma is essential in the proof of the moment formula in Theorem \ref{thm:moment-formula}. 
\begin{lemma}\label{lem:even-blocks}
For every noncrossing partition $\pi \in NC(p)$, 
$$1+\mathrm{e}(\pi) = \#(\pi \gamma).$$
\end{lemma}
\begin{proof}
We use a recurrence over the number of blocks of $\pi$. If $\pi$ has just one block, its associated geodesic permutation is $\gamma$ and one has
$$\#(\gamma^2) = \begin{cases}
1, &\quad \text{ if } p \text{ is odd;}\\ 
2, &\quad \text{ if } p \text{ is even.}\\
\end{cases}$$
For partitions $\pi$ with more than one block, we can assume without loss of generality that $\pi = \hat 1_k \sqcup \pi'$, where $\hat 1_k$ is a contiguous block of size $k$. Recall that the number of blocks of the (not necessarily geodesic) permutation $\pi\gamma$ is given by $ \#(\pi \gamma) = \#(\mathrm{fat}(\pi) \vee \rho_{14})$, where $\rho_{14} \in \Pi_2(2p)$ is the pair partition which pairs an element $i$ with $i+(-1)^{i+1}3$. 

If $k$ is an even number $k=2r$, then the partition $\mathrm{fat}(\hat 1_{2r} \sqcup \pi') \vee \rho_{14}$ contains the block $(1 \, 4 \, 5 \, 8 \, \cdots 4r-3 \, 4r)$ along with the blocks coming from elements of the form $4i+2, 4i+3$ from $\{1, \ldots, 4r\}$ and from $\pi'$. Since we are interested only in the number of blocks statistic (and not in the contents of each block), we can count the blocks of the join of two partitions by drawing them one beneath the other and counting the number of connected components of the curve, without taking into account the possible crossings, see Figure \ref{fig:fat-even}. We conclude that $\#(\mathrm{fat}(\pi) \vee \rho_{14}) = 1 + \#(\mathrm{fat}(\pi') \vee \rho'_{14})$ where $\rho'_{14}$ is $\rho_{14}$ restricted to the set $\{2k+1, 2k+2 \ldots, 2p\}$. If $k$ is odd, $k=2r+1$, there is no extra block appearing, so  $\#(\mathrm{fat}(\pi) \vee \rho_{14}) = \#(\mathrm{fat}(\pi') \vee \rho'_{14})$, see Figure \ref{fig:fat-odd}. 
\begin{figure}
\includegraphics{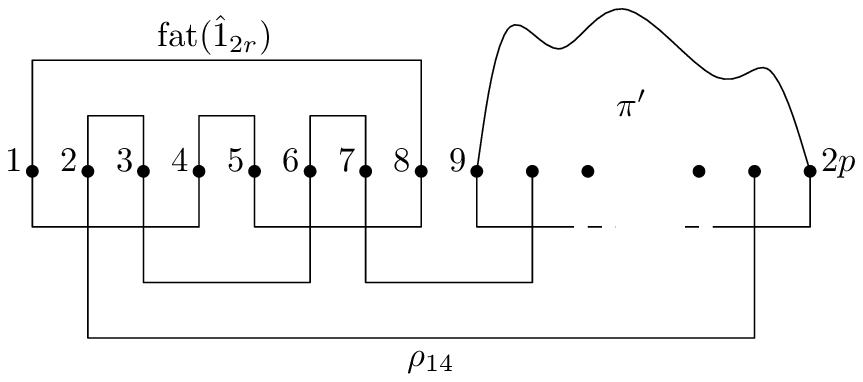} \\
\includegraphics{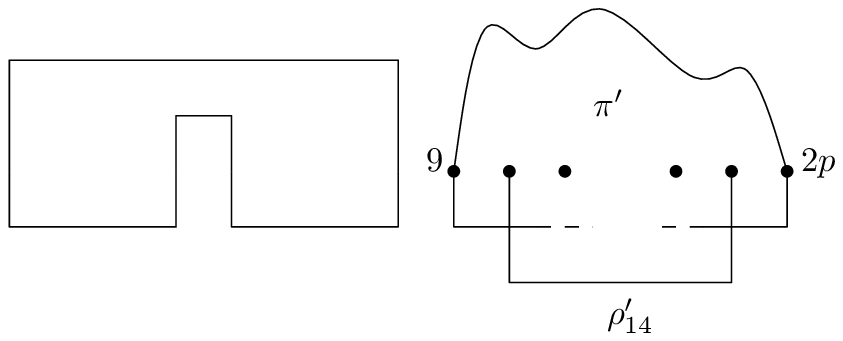}
\caption{The partition $\mathrm{fat}(\pi) \vee \rho_{14}$ (top line) has one more block than $\mathrm{fat}(\pi') \vee \rho'_{14}$ (bottom) when $k$ is even.}
\label{fig:fat-even}
\end{figure}
\begin{figure}
\includegraphics{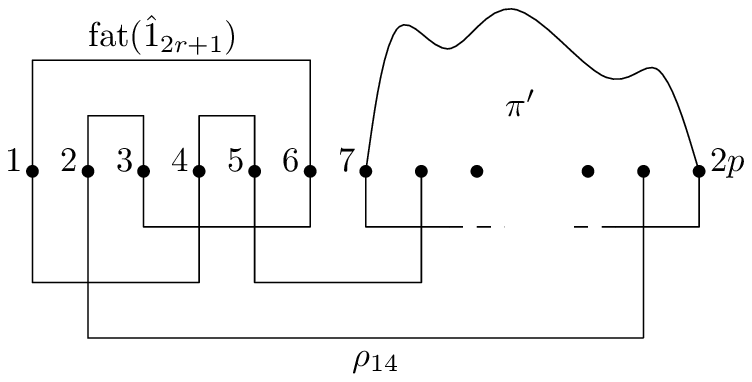} \qquad
\includegraphics{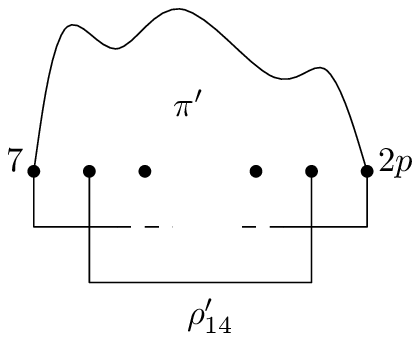}
\caption{The partition $\mathrm{fat}(\pi) \vee \rho_{14}$ (left) has the same number of blocks as $\mathrm{fat}(\pi') \vee \rho'_{14}$ (right) when $k$ is odd.}
\label{fig:fat-odd}
\end{figure}
\end{proof}

\begin{theorem}\label{thm:moment-formula}
For any $p\geq 1$ we have
$$\lim_{d\to\infty}(\mathbb E\circ \mathrm{tr})(m{W}^{\Gamma})^p
=\sum_{\pi \in NC(p)}m^{\#\pi}n^{\mathrm{e}(\pi)}$$
where $\#(\cdot)$ and $\mathrm{e}(\cdot)$ denote the number of blocks and the number of blocks of even size statistics. 
\end{theorem}

\begin{proof}
The matrix elements of the partial transpose matrix are given by:
$${W}^{\Gamma}_{ia,jb}=W_{ib,ja}=(dm)^{-1}\sum_{k=1}^d\sum_{c=1}^mG_{ib,kc}\bar{G}_{ja,kc}.$$

This gives:
\begin{align*}
\mathrm{tr}[({W}^{\Gamma})^p]
&=(dn)^{-1}(dm)^{-p}\sum_{i_1,\ldots,i_p=1}^d\sum_{a_1,\ldots,a_p=1}^n\prod_{s=1}^p{W}^{\Gamma}_{i_sa_s,i_{s+1}a_{s+1}}\\
&=(dn)^{-1}(dm)^{-p}\sum_{i_1,\ldots,i_p=1}^d\sum_{a_1,\ldots,a_p=1}^n\prod_{s=1}^p W_{i_sa_{s+1},i_{s+1}a_s} \\
&=(dn)^{-1}(dm)^{-p}\sum_{i_1,\ldots,i_p=1}^d\sum_{a_1,\ldots,a_p=1}^n\prod_{s=1}^p \sum_{j_1,\ldots,j_p=1}^d\sum_{b_1,\ldots,b_p=1}^mG_{i_sa_{s+1},j_sb_s}\bar{G}_{i_{s+1}a_s,j_sb_s}.
\end{align*}

After interchanging the product with the last two sums, the average of the general term can be computed by the Wick rule, namely:
$$\mathbb E\left(\prod_{s=1}^pG_{i_sa_{s+1},j_sb_s}\bar{G}_{i_{s+1}a_s,j_sb_s}\right)
=\sum_{\pi\in S_p}\prod_{s=1}^p\delta_{i_s,i_{\pi(s)+1}}\delta_{a_{s+1},a_{\pi(s)}}\delta_{j_s,j_{\pi(s)}}\delta_{b_s,b_{\pi(s)}}.$$

Let $\gamma\in S_p$ be the full cycle $\gamma=(1 \, 2 \, \ldots \, p)^{-1}$. The general factor in the above product is 1 if and only if the following four conditions are simultaneously satisfied:
$$\gamma^{-1}\pi\leq \ker i,\quad\pi\gamma \leq \ker a,\quad\pi \leq \ker j,\quad\pi \leq \ker b,$$
where the kernel of a function $f:\{1, \ldots, p\} \to \mathbb R$ is the partition $\ker f$ where $i$ and $j$ are in the same block iff $f(i) = f(j)$. Counting the number of free parameters in the above equation, we obtain:
\begin{align*}
(\mathbb E\circ \mathrm{tr})[({W}^{\Gamma})^p]
&=(dn)^{-1}(dm)^{-p}\sum_{\pi\in S_p}d^{\#\pi+\#(\gamma^{-1}\pi)}m^{\#\pi}n^{\#(\pi\gamma)}\\
&=\sum_{\pi\in S_p}d^{\#\pi+\#(\gamma^{-1}\pi)-p-1}m^{\#\pi-p}n^{\#(\pi\gamma)-1}.
\end{align*}

The exponent of $d$ in the last expression on the right is
\begin{align*}
N(\pi)
&=\#\pi+\#(\gamma^{-1}\pi)-p-1\\
&= p-1-(|\pi|+|\gamma^{-1}\pi|)\\
&= p-1-(|\pi|+|\pi^{-1}\gamma|)
\end{align*}
and, as explained in the beginning of this section, this quantity is known to be $\leq 0$, with equality iff $\pi$ is geodesic, hence associated to a noncrossing partition. We get:
$$(\mathbb E\circ \mathrm{tr})[({W}^{\Gamma})^p]=(1+O(d^{-1}))m^{-p}n^{-1}\sum_{\pi\in NC(p)}m^{\#\pi} n^{\#(\pi\gamma)}.$$
By Lemma \ref{lem:even-blocks}, we have $\#(\pi\gamma)=\mathrm{e}(\pi)+1$, which gives the result.
\end{proof}

\section{Moment generating function}\label{sec:mgf}

In this section we find an equation for the moment generating function of the asymptotic law of $m{W}^{\Gamma}$. This moment generating function is defined by:
$$F(z)=\lim_{d\to\infty}(\mathbb E\circ \mathrm{tr})\left(\frac{1}{1-zm{W}^{\Gamma}}\right).$$

Equivalently, $F(z)=\sum_{p=0}^\infty M_pz^p$, where $M_p$ is the moment computed in Theorem \ref{thm:moment-formula}. Note that one can crudely upper bound the moment $M_p$ by $4^p(mn)^p$ so that the moment generating function above is defined at least for $z$ small engouh, $|z| <(4mn)^{-1}$.

\begin{theorem}\label{thm:mgf}
The moment generating function of $m{W}^{\Gamma}$ satisfies the equation:
$$(F-1)(1-z^2F^2)=mzF(1+nzF)$$
\end{theorem}

\begin{proof}
We use Theorem \ref{thm:moment-formula}. If we denote by $N(p,b,e)$ the number of partitions in $NC(p)$ having $b$ blocks and $e$ even blocks, we have:
$$F=1+\sum_{p=1}^\infty\sum_{\pi\in NC(p)} z^pm^{\#\pi}n^{\mathrm{e}(\pi)}=1+\sum_{p=1}^\infty\sum_{b=0}^\infty\sum_{e=0}^\infty z^pm^bn^eN(p,b,e).$$

Let us try to find a recurrence formula for the numbers $N(p,b,e)$. If we look at the block containing $1$, this block must have $r\geq 0$ other legs, and we get:
\begin{align*}
N(p,b,e) 
&=\sum_{r\in 2\mathbb N}\sum_{p=\Sigma p_i+r+1}\sum_{b=\Sigma b_i+1}\sum_{e=\Sigma e_i}N(p_1,b_1,e_1)\ldots N(p_{r+1},b_{r+1},e_{r+1})\\
&+\sum_{r\in 2\mathbb N+1}\sum_{p=\Sigma p_i+r+1}\sum_{b=\Sigma b_i+1}\sum_{e=\Sigma e_i+1}N(p_1,b_1,e_1)\ldots N(p_{r+1},b_{r+1},e_{r+1}).
\end{align*}

Here $p_1,\ldots,p_{r+1}$ are the number of points between the legs of the block containing 1, so that we have $p=(p_1+\ldots+p_{r+1})+r+1$, and the whole sum is split over two cases ($r$ even or odd), because the parity of $r$ affects the number of even blocks of our partition.

Now by multiplying everything by a $z^pm^bn^e$ factor, and by carefully distributing the various powers of $z,m,b$ on the right, we obtain the following formula:
\begin{align*}
z^pm^bn^eN(p,b,e)
&=m\sum_{r\in 2\mathbb N}z^{r+1}\sum_{p=\Sigma p_i+r+1}\sum_{b=\Sigma b_i+1}\sum_{e=\Sigma e_i}\prod_{i=1}^{r+1}z^{p_i}m^{b_i}n^{e_i}N(p_i,b_i,e_i)\\
&+mn\sum_{r\in 2\mathbb N+1}z^{r+1}\sum_{p=\Sigma p_i+r+1}\sum_{b=\Sigma b_i+1}\sum_{e=\Sigma e_i+1}\prod_{i=1}^{r+1}z^{p_i}m^{b_i}n^{e_i}N(p_i,b_i,e_i).
\end{align*}

Let us sum now all these equalities, over all $p\geq 1$ and over all $b,e\geq 0$. According to the definition of $F$, at left we obtain $F-1$. As for the two sums appearing on the right (i.e., at right of the two $z^{r+1}$ factors), when summing them over all $p\geq 1$ and over all $b,e\geq 0$, we obtain in both cases $F^{r+1}$. So, we have the following formula:
\begin{align*}
F-1
&=m\sum_{r\in 2\mathbb N}(zF)^{r+1}+mn\sum_{r\in 2\mathbb N+1}(zF)^{r+1}\\
&=m\,\frac{zF}{1-z^2F^2}+mn\,\frac{z^2F^2}{1-z^2F^2}\\
&=mzF\,\frac{1+nzF}{1-z^2F^2}.
\end{align*}

This gives the formula in the statement.
\end{proof}

\section{Limiting distribution and free Poisson laws}\label{sec:main}

In this section we state and prove our main result. Let us first recall that for a complex Wishart matrix $W$ of parameters $(dn,dm)$, the eigenvalue distribution of $tW$, with $t=m/n$, converges in moments with $d\to\infty$ to the Marchenko-Pastur law \cite{mpa}, given by
$$\pi_t=\max (1-t,0)\delta_0+\frac{\sqrt{4t-(x-1-t)^2}}{2\pi x}1_{[a,b]}(x)\,dx,$$
where $a=(\sqrt t - 1)^2$ and $b=(\sqrt t +1)^2$.

In order to deal with the block-transposed Wishart matrices ${W}^{\Gamma}$, we will need a free probability point of view on $\pi_t$. So, let us recall from \cite{vdn} that a noncommutative probability space is a pair $(A,\varphi)$, where $A$ is a $C^*$-algebra, and $\varphi:A\to\mathbb C$ is a positive unital linear form. The law of a self-adjoint element $a\in A$ is the probability measure on the spectrum of $a$ (which is a compact subset of $\mathbb R$) having as moments the numbers $\varphi(a^p)$.

Two subalgebras $B,C\in A$ are called free if $\varphi(\ldots b_ic_ib_{i+1}c_{i+1}\ldots)=0$ whenever $b_i\in B$ and $c_i\in C$ satisfy $\varphi(b_i)=\varphi(c_i)=0$. Two elements $b,c\in A$  are called free whenever the algebras $B=<b>$ and $C=<c>$ that they generate are free. Finally, the free convolution operation $\boxplus$ is defined as follows: if $\mu,\nu$ are compactly supported probability measures on $\mathbb R$, then $\mu\boxplus\nu$ is the law of $b+c$, where $b,c$ are free, having laws $\mu,\nu$, see \cite{vdn}.

One of the remarkable results of the theory is that when performing a ``free Poisson limit'', we obtain precisely the Marchenko-Pastur law. That is, we have:
$$\pi_t=\lim_{k\to\infty}\left[\left(1-\frac{t}{k}\right)\delta_0+\frac{t}{k}\delta_1\right]^{\boxplus k}.$$

For this reason, the Marchenko-Pastur law is also called ``free Poisson law'' \cite{nsp}.

In what follows we will find a free probability interpretation of our main result so far, namely Theorem \ref{thm:mgf}. We will prove that $m{W}^{\Gamma}$ has the same law as $a-b$, where $a,b$ are free variables following the laws $\pi_s,\pi_t$, with $s=m(n+1)/2$ and $t=m(n-1)/2$. 

\begin{theorem}\label{thm:main}
With $d\to\infty$, the empirical spectral distribution of $m{W}^{\Gamma}$ converges in moments to a probability measure $\mu_{m,n}$, which is the free difference of free Poisson distributions of parameters $m(n\pm 1)/2$.
\end{theorem}

\begin{proof}
We use Voiculescu's $R$-transform, which is defined as follows \cite{vdn}. Let $\mu$ be a compactly supported probability measure on $\mathbb R$, having moment generating function $F$. The Cauchy transform of $\mu$ is then $G(\xi)=\xi^{-1}F(\xi^{-1})$. We denote by $K$ the formal inverse of $G$, given by $G(K(z))=K(G(z))=z$, and then we write $K(z)=R(z)+z^{-1}$. The function $R=R_\mu$, called $R$-transform of $\mu$, has the remarkable property $R_{\mu\boxplus\nu}=R_\mu+R_\nu$. In other words, $R$ is the free analogue of the logarithm of the Fourier transform.

Consider now the equation of $F$, found in Theorem \ref{thm:mgf}:
$$(F-1)(1-z^2F^2)=mzF(1+nzF).$$

With $z\to\xi^{-1}$ and $F\to\xi G$ (so that $zF\to G$) we obtain:
$$(\xi G-1)(1-G^2)=mG(1+nG).$$

Now with $\xi \to K$ and $G\to z$ we obtain:
$$(zK-1)(1-z^2)=mz(1+nz).$$

Finally, with $K\to R+z^{-1}$ we obtain:
$$zR(1-z^2)=mz(1+nz).$$

Thus the $R$-transform of the asymptotic law of $m{W}^{\Gamma}$ is given by:
$$R=m\,\frac{1+nz}{1-z^2}=\frac{m}{2}\left(\frac{n+1}{1-z}-\frac{n-1}{1+z}\right).$$

Let now $a,b$ be free variables, following free Poisson laws of parameters $s,t$. We have $R_a=s/(1-z)$ and $R_b=t/(1-z)$, and by using the general dilation formula $R_{qb}(z)=qR_b(qz)$  at $q=-1$ (see \cite{nsp}) we deduce that we have $R_{-b}=-t/(1+z)$. So, we have:
$$R_{a-b}=\frac{s}{1-z}-\frac{t}{1+z}.$$

Now since with $s=m(n+1)/2$ and $t=m(n-1)/2$ we obtain the above formula of $R$, we are done.
\end{proof}

\section{Examples of limiting measures}\label{sec:examples}

In the reminder of this paper we study the limiting measures found in Theorem \ref{thm:main}. It is convenient to enlarge our study to all the possible values of $m,n$, as follows.

\begin{definition}\label{41}
To any real numbers $m\geq 0$ and $n\geq 1$ we associate the measure $\mu_{m,n}$ which is the free difference of free Poisson distributions of parameters $m(n\pm 1)/2$.
\end{definition}

As a first remark, the various formulae found in this paper (e.g. moment formula in Theorem \ref{thm:moment-formula}, and various equations for $F,G,K,R$ in the proof of Theorem \ref{thm:main}) are of valid in this more general setting. Let us collect some useful formulae here.

\begin{proposition}\label{42}
The measure $\mu_{m,n}$ has the following properties:
\begin{enumerate}
\item It has mean $m$ and variance $mn$.

\item Its Cauchy transform satisfies $(\xi G-1)(1-G^2)=mG(1+nG)$.

\item Its $R$-transform is given by $R(z)=m(1+nz)/(1-z^2)$.
\end{enumerate}
\end{proposition}

\begin{proof}
The first statement follows from the fact that both the average and the variance of a free Poisson law of parameter $t$ are $t$. The last two items follow from the proof of Theorem \ref{thm:main} above, by proceeding backwards.
\end{proof}

Regarding now the main particular cases of measures of type $\mu_{m,n}$, corresponding to some previously known random matrix computations, the situation is as follows.

\begin{theorem}\label{43}
The measures $\mu_{m,n}$ are as follows:
\begin{enumerate}
\item At $n=1$ we have the Marchenko-Pastur laws: $\mu_{m,1}=\pi_m$.

\item With $m=\alpha n\to\infty$ we obtain Aubrun's shifted semicircles \cite{aub}.

\item With $m=t/n\to 0$ we obtain a modified free Bessel distribution $\tilde \pi_{2, t}$.
\end{enumerate}
\end{theorem}

\begin{proof}
All the results are clear either from Definition \ref{41}, or from Proposition \ref{42}:

(1) This follows from Definition \ref{41}, because at $n=1$ we have $m(n\pm 1)/2=m,0$.

(2) The precise statement here is that for any $\beta=\alpha^{-1}>0$ we have the formula
$$\lim_{m\to\infty}D_{m^{-1}}(\mu_{m,\beta m})=SC(1,\beta)$$
where $D$ is the dilation operation (if $\mu$ is the law of a random variable $a$ then $D_q(\mu)$ is by definition the law of the random variable $qa$) and where $SC(1,\beta)$ is the semicircle law centered at $1$, having support $[1-2\sqrt{\beta},1+2\sqrt{\beta}]$.

Indeed, let $\Gamma=\Gamma(\xi)$ be the Cauchy transform of $\nu=D_{m^{-1}}(\mu_{m,n})$. The Cauchy transform of $\mu_{m,n}=D_m(\nu)$ is then given by $G(m\xi)=m^{-1}\Gamma$, and by making the replacements $\xi\to m\xi$ and $G\to m^{-1}\Gamma$ in the equation of $G$ in Proposition \ref{42}, we obtain: 
$$(\xi\Gamma-1)(1-m^{-2}\Gamma^2)=\Gamma(1+nm^{-1}\Gamma).$$

In the limit $n=\beta m\to\infty$, as indicated above, this equation becomes:
\begin{align*}
\xi\Gamma-1=\Gamma(1+\beta\Gamma)
&\implies\beta\Gamma^2+(1-\xi)\Gamma+1=0\\
&\implies\Gamma=\frac{\xi-1\pm\sqrt{(1-\xi)^2-4\beta}}{2\beta}.
\end{align*}

By applying now the Stieltjes inversion formula, we get the following density:
$$f(x)=\frac{\sqrt{4\beta-(1-x)^2}}{2\beta\pi}$$

But this is exactly the density of the semicircle law $SC(1,\beta)$, and we are done.

(3) This follows directly from Definition \ref{41}, because with $m=t/n\to 0$ we have $m(n\pm 1)/2=t/2,t/2$, and these parameters are the same as those for the modified free Bessel law $\tilde \pi_{2, t}$, see \cite{bb+}, Definition 7.2.
\end{proof}

Observe that the formula (1) above is in tune with the Marchenko-Pastur theorem, as formulated in section \ref{sec:main}: indeed, at $n=1$ the partial transposition operation is trivial, so Theorem \ref{thm:main} computes the asymptotic law of the rescaled Wishart matrix $mW$.

Also, (2) above agrees of course with Aubrun's computation in \cite{aub}, with Theorem \ref{thm:main} generalizing Aubrun's result. This result can be seen as well to follow from the following heuristic argument using the free central limit theorem. Using the free Poisson limit theorem, one can write (we denote a Bernoulli random variable by $b(p, x) = (1-p) \delta_0 + p \delta_x$)
\begin{align*}
D_{m^{-1}}(\mu_{m,n}) &= \left[\lim_{k \to \infty}b\left(\frac{m(n-1)}{2k}, -\frac{1}{m}\right)^{\boxplus k}\right] \boxplus \left[\lim_{k \to \infty}b\left(\frac{m(n+1)}{2k}, \frac{1}{m}\right)^{\boxplus k}\right] \\
&=\delta_1 \boxplus \lim_{k \to \infty}\left[b\left(\frac{m(n-1)}{2k}, -\frac{1}{m}\right) \boxplus b\left(\frac{m(n-1)}{2k}, -\frac{1}{m}\right) \boxplus \delta_{-1/k}\right]^{\boxplus k}\\
&=\delta_1 \boxplus \lim_{k \to \infty}\left[D_{k^{-1/2}}(\tilde \mu_{m,n,k})\right]^{\boxplus k}
\end{align*}
where
$$\tilde \mu_{m,n,k} = \delta_{-1/\sqrt{k}} \boxplus b\left(\frac{m(n-1)}{2k}, -\frac{\sqrt{k}}{m}\right) \boxplus b\left(\frac{m(n+1)}{2k}, \frac{\sqrt{k}}{m}\right)$$
is a centered random variable of variance $n/m$. In the regime when $n=\beta m$, $k=m^2$ and $m \to \infty$, the above measure converges weakly to a free difference of two identical Bernoulli random variables of parameter $\beta/2$ and the result follows by the free central limit theorem.

As for the formula (3) above, this is not exactly the random matrix result in \cite{bb+}, because the matrices studied there are of quite different nature of those studied here. We intend to come back to this phenomenon in some future work \cite{bne}. 

\section{Properties of the limiting measure}\label{sec:prop} 

In this section we study the atoms, support and density of the measures $\mu_{m,n}$. These questions are of purely probabilistic nature, and our first result here is as follows.

\begin{proposition}\label{51}
The probability measure $\mu_{m,n}$ has at most one atom, at $0$, of mass $\max\{1-mn, 0\}$.
\end{proposition}

\begin{proof}
The result follows from the general characterization of atoms of a free additive convolution, given by  Bercovici and Voiculescu in \cite{bvo}: $x$ is an atom for the free additive convolution of two measures $\nu_1$ and $\nu_2$ iff $x=x_1+x_2$, where $\nu_1(\{x_1\}) + \nu_2(\{x_2\}) > 1$. And, in addition, if this is the case, then $[\nu_1 \boxplus \nu_2](\{x\}) = \nu_1(\{x_1\}) + \nu_2(\{x_2\}) - 1$. 

In our situation, since the free Poisson distribution of parameter $c$ is known to have  an atom at $0$ of mass $\max\{1-c, 0\}$, we can apply the above-mentioned result when $\nu_1$ is a free Poisson distribution of parameter $m(n+1)/2$ and $\nu_2$ is the image of a free Poisson distribution of parameter $m(n-1)/2$ through the negation map, and we are done.
\end{proof}

Regarding now the support and density of $\mu_{m,n}$, our main tool will be of course the degree 3 equation for the Cauchy transform appearing in Proposition \ref{42}, namely:
$$(\xi G-1)(1-G^2)=mG(1+nG).$$

In principle all the needed information can be obtained from these equations via the Stieltjes inversion formula, and various calculus methods for small degree polynomials.

In what follows we will investigate a question which is particular interest for quantum information theory, as explained by Aubrun in \cite{aub}: the problem is that of deciding, for various values of $m\geq 0$ and $n\geq 1$, whether the support is contained in $[0,\infty)$ or not. The answer here is as follows (for a graphical representation, see Figure \ref{fig:support} below).

\begin{figure}
\includegraphics{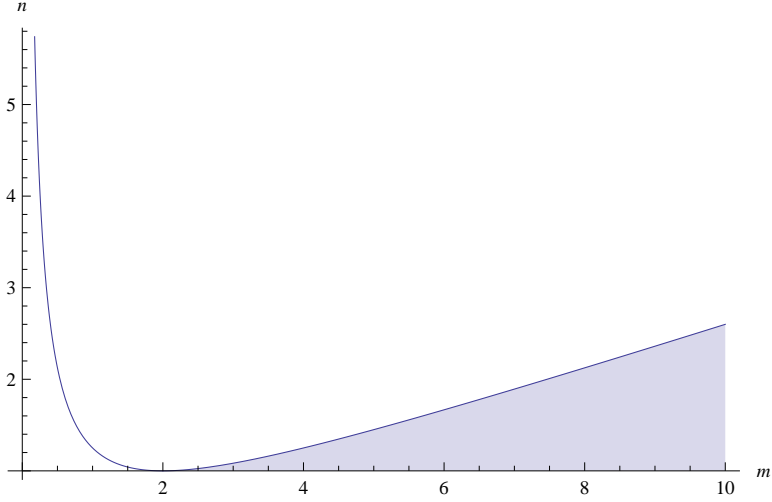}
\caption{Positivity of the support of $\mu_{m,n}$. Measures with parameters belonging to the shaded region ($m \geq 2$, $1 \leq n \leq m/4+1/m$) have positive support.}
\label{fig:support}
\end{figure}

\begin{theorem}\label{thm:support}
The measure $\mu_{m,n}$ has positive support iff $n\leq m/4+1/m$ and $m \geq 2$.
\end{theorem}

\begin{proof}
We use the fact that the support of the a.c. part of the probability measure $\mu_{m,n}$ is a union of intervals defined by the points where the analyticity of $G$ breaks. These points are the roots of the discriminant of the equation defining $G$ (seen as an equation in $G$):
\begin{align*}
\Delta(\xi) &= 4 \xi^4-12 m \xi^3+(n^2 m^2+12 m^2-20 n m-8) \xi^2+\\
&\qquad +(-2 n^2 m^3-4 m^3+22 n m^2-20 m) \xi+\\
&\qquad +n^2 m^4-4 n^3 m^3-2 n m^3+12 n^2 m^2+m^2-12 n m+4.
\end{align*}
This is a degree 4 equation in $\xi$ that has either 2 or 4 real solutions (4 complex solutions is not a possibility since the support has to be non-empty). The discriminant of $\Delta(x)$ allows to decide between these cases:
\begin{align*}
\Delta_2(m,n) = -256 m^2 (n-1) (n+1) \left(m^3 n^3+15 m^2 n^2-27 m^2+48 m n-64\right)^3.
\end{align*}
The sign of the discriminant of a quartic equation permits to decide whether the equation has two real and two complex roots or 4 roots of the same type (real or complex): the discriminant is negative iff the quartic has two real and two complex roots. In our case, the sign of $\Delta_2$ is the opposite of the sign of the bivariate polynomial 
\begin{align*}
P(m,n) = m^3 n^3+15 m^2 n^2+48 m n-27 m^2-64.
\end{align*}
Using straightforward calculus, it is easy to see that in the domain of interest here ($m>0$, $n>1$), $P(m, n)<0$ if and only if $m<4$ and $n < g(m)$ where $g(m)$ is the only real root of $P$ (seen as a polynomial in $n$):
$$g(m)\!=\!\frac{1}{2m}\left[3\cdot 2^{2/3} \left(2+m^2+m \sqrt{4+m^2}\right)^{1/3}\!\!\!\!+6\cdot 2^{1/3}\left(2+m^2+m \sqrt{4+m^2}\right)^{-1/3}-10\right].$$
The function $g$ is larger than one for $m \in [0,4]$ and it is decreasing in $m$ in that range; its curve is plotted in Figure \ref{fig:support-proof}. This settles the question of deciding if the support of $\mu_{m,n}$ has one or two disjoint intervals: measures corresponding to parameters in regions A1, A2 and B in Figure \ref{fig:support-proof} have support made of 2 disjoint intervals, whereas regions C and D correspond to measures having connected absolutely continuous support. 

\begin{figure}
\includegraphics{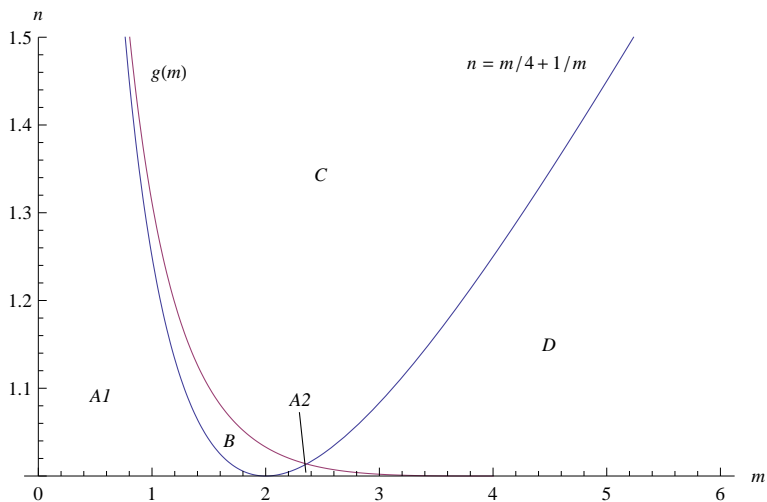}
\caption{Behavior of the support of $\mu_{m,n}$. Regions A1, A2 and B correspond to measures having 2 disjoint intervals for the a.c. support, whereas regions C and D have connected a.c. support. Measures with parameters in A1 have a.c. support made of two intervals separated by the origin and measures in A2 have positive support.}
\label{fig:support-proof}
\end{figure}

Another indicator for positive support is the value at the origin $\xi=0$ of the discriminant $\Delta$:
$$\Delta(0) = (m^2 - 4 m n + 4) (m n - 1)^2.$$
The two simple curves $g(m)$ and $h(m) = m/4+1/m$ partition the set $(0, \infty) \times (1, \infty)$ into five connected components, A1, A2, B, C, D, see Figure \ref{fig:support-proof}. Points in the regions A1, A2 and D have $\Delta(0) >0$, whereas points in the regions B and C have $\Delta(0)<0$. A negative value at the origin indicates (since the leading coefficient of $\Delta(\xi)$ is strictly positive) that zero is in the interior of the support of $\mu_{m,n}$. Hence, positively supported measures correspond to a subset of $A1 \cup A2 \cup D$. Moreover, since the average of $\mu_{m,n}$ is positive and measures in D have connected support, it follows that all the measures in D are positively supported. The unique point of intersection of the graphs of $g$ and $h$ in the domain $m >0$, $n>1$ is $n=\sqrt{-\frac{5}{3}+\frac{14}{3 \sqrt{3}}} \approx 1.01372$, $m=-\frac{9}{77} \left(-23 n+3 n^3\right) \approx 2.35992$.

To conclude, one needs to understand how measures with parameters in A1 and A2 behave. We claim that measures in A1 are supported on $I_1 \cup I_2$ where $I_1 \subset \mathbb R_-$, $I_2 \subset \mathbb R_+$ and measures in A2 are supported on $I_1 \cup I_2$ where $I_1, I_2 \subset \mathbb R_+$. 

To show this, note that both regions A1 and A2 are below $g$ and $h$, hence the support has two connected components $I_1, I_2$ and it does not contain zero. Without losing generality, we assume that $I_1$ is at the left of $I_2$. Three cases need to be considered: zero is at the left of $I_1$, zero separates  $I_1$ and $I_2$ or zero is at the right of $I_2$. The last case cannot occur, since $\mu_{m,n}$ cannot be supported entirely on the negative half line (it has positive mean). To distinguish between the first two cases, we analyze the derivatives of $\Delta$, given by
\begin{align*}
\Delta'(\xi) &= 16 \xi^3-36 m \xi^2+\left(2 n^2 m^2+24 m^2-40 n m-16\right) \xi-\\
&\qquad -2 n^2 m^3-4 m^3+22 n m^2-20 m;\\
\Delta''(\xi) &= 48 \xi^2-72 m \xi+2 n^2 m^2+24 m^2-40 n m-16.
\end{align*}
The sign of the roots of $\Delta'$ allows to distinguish between the two cases above: $\Delta$ has 4 positive roots if and only if $\Delta'$ has 3 positive roots. The other case where $\Delta$ has two positive and two negative roots corresponds to $\Delta'$ having at least one (and at most two) negative roots. Since both $\Delta'$ and $\Delta''$ have negative subleading coefficient, it follows that these polynomials must have at least one positive root. In  Table \ref{tbl:roots}, the nature of the roots of $\Delta'$ is given in terms of the sign of $\Delta'(0)$ and $\Delta''(0)$.

\begin{table}
\begin{tabular}{r|l}
Signs of $\Delta'(0)$, $\Delta''(0)$ & Nature of the roots of $\Delta'$\\
\hline
$\Delta'(0)<0$, $\Delta''(0)<0$ & 2 negative, 1 positive \\
$\Delta'(0)<0$, $\Delta''(0)>0$ & 3 positive \\
$\Delta'(0)>0$, $\Delta''(0)<0$ & 1 negative, 2 positive \\
$\Delta'(0)>0$, $\Delta''(0)>0$ & 1 negative, 2 positive\\
\end{tabular}
\caption{Nature of the roots of $\Delta'(\xi)$ in terms of the signs of $\Delta'(0)$ and $\Delta''(0)$.}
\label{tbl:roots}
\end{table}

Both equations $\Delta'(0)=0$ and $\Delta''(0)=0$ are of second degree in $n$, with solutions
$$n=p_{1,2}(m) = \frac{11\pm\sqrt{81-8 m^2}}{2 m} \qquad m \in (0, 9/\sqrt{8}]$$
and respectively
$$n=q_{1,2}(m) = \frac{10 \pm 2\sqrt{3} \sqrt{9-m^2}}{m} \qquad m \in (0, 3],$$
hence $\Delta'(0)>0$ if and only if $m \in (0, 9/\sqrt{8})$ and $p_1 < m < p_2$ and $\Delta''(0)<0$ if and only if $m \in (0, 3)$ and $q_1 < m < q_2$. These functions are plotted in Figure \ref{fig:pqh}. 

Straightforward computations show that $p_1(m)\geq 1$ for $m \in (0, 5/3] \cup [2, 9/\sqrt{8}]$ and that $h(m)>p_1(m)$ for $m \in (0,2)$. The equation $q_1(m)=1$ has a unique positive solution $5/3 < 2(5+\sqrt{51})/13 < 2$. From these facts, one can conclude that the only couples $(m,n)$ for which $\Delta'(0)<0$ and $\Delta''(0)>0$ (or, equivalently, the equation $\Delta(\xi)=0$ has 4 positive real roots) and which are below the graph of $h$ (i.e. satisfy $n < m/4+1/m$) are those for which $m>2$, which concludes the proof.
\end{proof}

\begin{figure}
\includegraphics{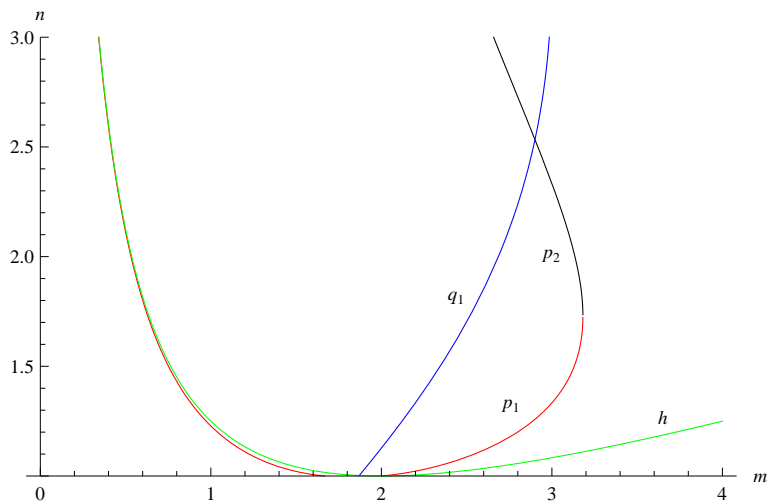}
\caption{Plots of the solutions for the equations $\Delta'(0)=0$ and $\Delta''(0)=0$, along with $h(m)$.}
\label{fig:pqh}
\end{figure}

\section{Concluding remarks}\label{sec:final}

We have seen in this paper that Aubrun's recent and surprising result in \cite{aub} on the block-transposed Wishart matrices, which basically says that ``when performing the block transposition the Marchenko-Pastur law becomes a shifted semicircle law'' is best understood first by enlarging the set of parameters, and then by using free probability theory. Indeed, by replacing Aubrun's parameters $(d^2,\alpha d^2)$ by general parameters $(dn,dm)$, and then by doing a free probabilistic study, our conclusion is that ``when performing the block transposition, the free Poisson law becomes a free difference of free Poisson laws''. 

With this new point of view, several problems appear. First, since both in the usual Wishart and in the block-transposed Wishart cases we simply reach to certain ``free linear combinations of free Poisson laws'', one may wonder about a general result, covering both the Wishart and block-transposed Wishart cases. In addition, the random matrix result in \cite{bb+}, once again dealing with some ``block-modified'' Wishart matrices, and once again leading to certain free combinations of free Poisson laws, is waiting as well to be generalized. There are probably some connections as well with the work of Lenczewski in \cite{len}. We intend to come back to these algebraic questions, and to perform as well a systematic study of the resulting asymptotic measures, in a forthcoming paper \cite{bne}.

As explained in the introduction, it is of interest for quantum information theory if the partial transpose of a typical matrix (or quantum state) is positive or not. For a fixed choice of parameters $m$ and $n$, having an asymptotic eigenvalue counting measure which is supported on the positive half line is an indication that typical states are PPT. However, in order to decide if a quantum state is PPT or not, one needs to look at the smallest eigenvalue of the corresponding matrix. Aubrun has shown in \cite{aub} that the smallest eigenvalue converges indeed to the infimum of the support of the asymptotic measure; we conjecture that the same holds true in our setting.

\end{document}